\documentclass[printscheme,screen]{aomart}

\usepackage{amssymb,mathtools}
\usepackage[shortlabels]{enumitem}
\usepackage{microtype}
\usepackage{needspace}
\usepackage{xurl}
\usepackage{booktabs,array}
\copyrightnote{}
\doinumber{}
\fancypagestyle{firstpage}{\fancyhf{}\fancyfoot[C]{\thepage}}

\numberwithin{equation}{section}
\newcommand{\R}{\mathbb R}
\newcommand{\N}{\mathbb N}
\newcommand{\BR}{\operatorname{BR}}
\newcommand{\cL}{\mathcal L}

\theoremstyle{plain}
\newtheorem{theorem}{Theorem}[section]
\newtheorem{proposition}[theorem]{Proposition}
\newtheorem{lemma}[theorem]{Lemma}

\theoremstyle{definition}
\newtheorem{definition}[theorem]{Definition}
\theoremstyle{remark}

\renewcommand{\theequation}{%
  \ifnum\value{section}=1
    \arabic{equation}%
  \else
    \thesection.\arabic{equation}%
  \fi}
\renewcommand{\thetheorem}{%
  \ifnum\value{section}=1
    \arabic{theorem}%
  \else
    \thesection.\arabic{theorem}%
  \fi}

\title[Bounded Ratios of Lorentzian Polynomials II]
{Bounded Ratios of Lorentzian Polynomials\texorpdfstring{\\}{ }
II: Quadratic Derivatives}

\author{Dijia Chen}
\address{Department of Mathematics, University of Wisconsin--Madison, Madison, WI, USA}
\email{dchen426@wisc.edu}
\author{Bowen Gan}
\address{Institute of Mathematical Sciences, ShanghaiTech University, Shanghai, China}
\email{ganbw2023@shanghaitech.edu.cn}
\author{Ivy Liu}
\address{Department of Mathematics, University of Wisconsin--Madison, Madison, WI, USA}
\email{ivy.liu@wisc.edu}
\author{Zemeng Wang}
\address{Department of Mathematics, University of Wisconsin--Madison, Madison, WI, USA}
\email{zwang3694@wisc.edu}
\author{Chengzhi Wu}
\address{Department of Mathematics, University of Wisconsin--Madison, Madison, WI, USA}
\email{cwu493@wisc.edu}

\hypersetup{pdftitle={Bounded Ratios of Lorentzian Polynomials II: Quadratic Derivatives},
pdfauthor={Dijia Chen, Bowen Gan, Ivy Liu, Zemeng Wang, Chengzhi Wu}}

\begin{document}

\begin{abstract}
We study bounded monomial ratios for Lorentzian polynomials. We prove that, in degree $d\geq2$ and $n$ variables, the bounded ratios from quadratic derivatives generate all bounded ratios for every $M$-convex support if and only if $n\leq3$, $d=2$, or $(n,d)=(4,3)$. For four-variable cubics, the proof uses a decomposition into $M$-concave functions. We construct explicit counterexamples with full support for four-variable quartics and five-variable cubics, then extend these two examples to all remaining pairs.
\end{abstract}

\maketitle
\thispagestyle{empty}
\raggedbottom
\tableofcontents

\section{Introduction and statements of the main theorems}

The \emph{Alexandrov--Fenchel} inequality gives multiplicative inequalities among \emph{mixed volumes}. For convex bodies $K_1,\ldots,K_n\subset\mathbb{R}^d$, the polynomial $\operatorname{Vol}(x_1K_1+\cdots+x_nK_n)$ records their mixed volumes in its coefficients. Fixing $d-2$ of the mixed-volume arguments gives a symmetric matrix $P=(p_{ij})$ satisfying $p_{ii}p_{jj}\le p_{ij}^2$ \cite[Chapter~7]{Schneider}. Br\"and\'en and Huh \cite{BH} introduced Lorentzian polynomials as a common framework for these and other examples from geometry and combinatorics. We study multiplicative coefficient inequalities for the class of Lorentzian polynomials.

Huang, Huh, Soskin, and Wang call $\ \ \displaystyle \prod_{i\le j}p_{ij}^{\alpha_{ij}}\ $ a \emph{bounded ratio} if, for fixed real exponents $\alpha_{ij}$, it is uniformly bounded above as $P=(p_{ij})$ ranges over \emph{Lorentzian matrices} with positive entries \cite[Definition~1.1]{HHSW}. They characterized these ratios by identifying the cone of \emph{reduced bounded ratios} with the dual of the \emph{cut cone} \cite[Theorem~B]{HHSW}. The generators in small sizes are particularly explicit. In sizes three and four, every bounded ratio is a product of nonnegative real powers of the \emph{triangular ratios}
\[
 \frac{p_{ii}p_{jk}}{p_{ij}p_{ik}},
\]
where $i,j,k \text{ are pairwise distinct}$. Each triangular ratio is bounded above by $2$. For matrices of size five, one also needs the \emph{pentagonal ratios}. They are obtained by permuting the indices in
\[
 \frac{p_{12}p_{13}p_{23}p_{44}p_{45}p_{55}}
      {p_{14}p_{15}p_{24}p_{25}p_{34}p_{35}},
\]
and each is bounded above by $4$. Together with the triangular ratios, they generate every bounded ratio in size five \cite[Example~1.4 and Theorem~A]{HHSW}.

Every nonzero quadratic derivative of a Lorentzian polynomial is Lorentzian \cite[Corollary~2.11]{BH}. Its normalized coefficients come from the original polynomial with shifted indices, so any bounded ratio for the derivative gives a bounded ratio for the original polynomial. We ask whether every bounded ratio is a product of nonnegative real powers of ratios obtained in this way.

A homogeneous polynomial with nonnegative coefficients and $M$-convex support is Lorentzian if and only if each of its nonzero quadratic derivatives is Lorentzian~\cite[Definition~2.6 and Theorem~2.25]{BH}. Part~I~\cite[Theorem~3.6]{CGLWW} shows that for Lorentzian polynomials in three variables with a fixed $M$-convex support, every bounded ratio is a product of nonnegative real powers of bounded ratios obtained from quadratic derivatives. Independently, Bathija, Rohatgi, and Soskin \cite[Corollary~3.23 and Theorem~8.3]{BRS} prove that triangular ratios generate the cone of bounded ratios for ternary forms of every degree with full support and for four-variable cubics with full support. They also give a dual description of the cone for full support \cite[Theorem~A]{BRS}. Baldi and Kummer \cite[Theorem~1.1 and Corollary~1.2]{BaldiKummer2026} give a dual description for every fixed $M$-convex support, using tropicalization and discrete convexity.

The definitions of $M$-convex sets and Lorentzian polynomials are recalled in Section~\ref{sec:preliminaries}. Now we fix $n\geq1$ and $d\geq2$, and let $\Delta_n^d=\{\alpha\in\mathbb Z_{\geq0}^n:|\alpha|=d\}$. For a nonempty $M$-convex set $S\subseteq\Delta_n^d$, let $\cL(S)$ be the family of Lorentzian polynomials of the form
\[
 f(\mathbf x)=\sum_{\alpha\in S}
 \frac{c_\alpha}{\alpha!}\,\mathbf x^\alpha,
\]
where $c_\alpha>0,\ \mathbf x^\alpha=x_1^{\alpha_1}\cdots x_n^{\alpha_n}$ and $\alpha!=\alpha_1!\cdots\alpha_n!$. The cone $\BR(S)\subseteq\mathbb R^S$ consists of the exponent vectors $w$ for which $\prod_{\alpha\in S}c_\alpha^{\,w_\alpha}$ is uniformly bounded in $\cL(S)$. Let $K(S)$ be the cone generated by the exponent vectors of all bounded ratios on $\cL(S)$ obtained from nonzero quadratic derivatives.

Our main result determines all pairs $(n,d)$ for which $K(S)=\BR(S)$ for every nonempty $M$-convex support $S$.

\begin{theorem}
\label{thm:main-classification}
Suppose that $n\geq1$ and $d\geq2$. One has $\BR(S)=K(S)$ for every nonempty $M$-convex support $S\subseteq\Delta_n^d$ if and only if
\[
 n\leq3,\quad d=2,\quad\text{or}\quad(n,d)=(4,3).
\]
For full support, we write $K(n,d)=K(\Delta_n^d)$ and $\BR(n,d)=\BR(\Delta_n^d)$. For every remaining pair $(n,d)$, one has
\[
 K(n,d)\subsetneq\BR(n,d).
\]
\end{theorem}

The proof for four-variable cubics uses convex separation, as in Part~I~\cite[proof of Theorem~3.6]{CGLWW}. We express the separating functions as nonnegative combinations of $M$-concave functions. For the two counterexamples, we obtain bounded ratios by combining inequalities from quadratic derivatives that share coefficients. Separating functions then show that these ratios are not generated by the bounded ratios of the individual derivatives. We extend the examples to higher degrees by shifting coefficient indices and to more variables by summing coordinates. 

The rest of the paper is organized as follows. Section~\ref{sec:preliminaries} recalls the definitions and preliminary results. Section~\ref{sec:quaternary-cubics} treats four-variable cubics. Sections~\ref{sec:minimal-obstructions} and \ref{sec:propagation} give the counterexamples and complete the classification.

\vspace{0.8\baselineskip}

\noindent\textbf{Acknowledgments.}
This project grew out of a reading group on Lorentzian polynomials organized by Botong Wang. We thank him for introducing us to the subject and for his guidance and encouragement.
\section{Preliminaries}
\label{sec:preliminaries}

\subsection{Lorentzian polynomials and discrete concavity}

Throughout the paper, we take $n\geq1$. We adopt the notation of \cite[Section~2.1]{BH} and denote the set of nonnegative integers by $\N$, and let $[n]=\{1,\ldots,n\}$. We define the \textit{d-th discrete simplex} $\Delta_n^d=\{\alpha\in\N^n:|\alpha|=d\}$ for $d\geq0$, where $|\alpha|=\sum_i\alpha_i$. We let $\{e_1,\ldots,e_n\}$ denote the standard basis of $\R^n$. For a multi-index $\alpha$, we set $\alpha!=\prod_i\alpha_i!$, $\mathbf x^\alpha=\prod_i x_i^{\alpha_i}$, and $\partial^\alpha=\prod_i\partial_i^{\alpha_i}$, where $\partial_i=\partial/\partial x_i$.

We use the factorial normalization of \cite[Section~3.3]{BH}. A nonzero homogeneous polynomial $f$ of degree $d\geq1$ with nonnegative real coefficients has the expansion
\begin{equation*}
 f(\mathbf x)=\sum_{\alpha\in\Delta_n^d}
                  \frac{c_\alpha}{\alpha!}\,\mathbf x^\alpha,
\end{equation*}
where each $c_\alpha$ is nonnegative. The \emph{Hessian} of $f$ is the symmetric matrix
\[
\mathcal{H}_f(\mathbf x)
=
\bigl(\partial_i\partial_j f\bigr)_{i,j=1}^{n}.
\] 
The support of $f$ is $S=\operatorname{supp}(f)=\{\alpha\in\Delta_n^d:c_\alpha>0\}$.

Br\"and\'en and Huh \cite{BH} observed that the supports of stable polynomials of degree 2 share a combinatorial property known as \textit{M-convexity}. They used the symmetric exchange formulation \cite[Section~4.1]{Murota} in the arguments about supports as below.
\begin{definition}
A nonempty set $S\subseteq\Delta_n^d$ is \emph{$M$-convex} if, for every $\alpha,\beta\in S$ and every $i$ with $\alpha_i>\beta_i$, there is an index $j$ with $\alpha_j<\beta_j$ such that both $\alpha-e_i+e_j$ and $\beta+e_i-e_j$ belong to $S$.
\end{definition}

Under this normalization, differentiation shifts the coefficient indices. In other words, for $d\geq2$ and $\beta\in\Delta_n^{d-2}$, the quadratic derivative has the expansion
\begin{equation}\label{eq:quadratic-derivative-expansion}
  \partial^\beta f(\mathbf x)=
 \sum_{\eta\in\Delta_n^2}\frac {c_{\beta+\eta}}{\eta!}\,\mathbf x^\eta.   
\end{equation}

Its support $S_\beta$ and Hessian matrix $\mathcal{H}_\beta(f)$ are
\[
 S_\beta=\{\eta\in\Delta_n^2:\beta+\eta\in S\},\qquad
 \mathcal{H}_\beta(f)=\bigl(c_{\beta+e_i+e_j}\bigr)_{i,j\in[n]}.
\]
If $S$ is $M$-convex, then every nonempty $S_\beta$ is $M$-convex. Indeed, for $\eta,\xi\in S_\beta$, symmetric exchange applied to $\beta+\eta$ and $\beta+\xi$ gives the required exchange in $S_\beta$ after subtraction of $\beta$.

Br\"and\'en and Huh \cite[Definition~2.6 and Theorem~2.25]{BH} showed that Lorentzian polynomials can be described in terms of M-convex supports and the Hessians of their quadratic derivatives. We use these two conditions in the following definition.

\begin{definition}\label{def:LP}
A nonzero homogeneous polynomial $f$ of degree $d\geq2$ with nonnegative coefficients is \emph{Lorentzian} if its support is $M$-convex and, for every $\beta\in\Delta_n^{d-2}$, the matrix $\mathcal{H}_\beta(f)$ has at most one positive eigenvalue. The zero polynomial and every homogeneous linear polynomial with nonnegative coefficients are also regarded as Lorentzian.
\end{definition}

For a nonempty $M$-convex $S\subseteq\Delta_n^d$ with $d\geq1$, we denote by $\cL(S)$ the class of Lorentzian polynomials with support exactly $S$ and set $\cL(n,d)=\cL(\Delta_n^d)$.

For functions on an $M$-convex set, the corresponding exchange condition is $M$-concavity \cite[Section~6.1]{Murota}. These functions give families of Lorentzian polynomials and hence necessary conditions for bounded ratios.
\begin{definition}
Suppose that $S\subseteq\Delta_n^d$ is nonempty and $M$-convex. A real-valued function $q$ on $S$ is \emph{$M$-concave} if, for every $\alpha,\beta\in S$ and every $i$ with $\alpha_i>\beta_i$, there is an index $j$ with $\alpha_j<\beta_j$ such that both $\alpha-e_i+e_j$ and $\beta+e_i-e_j$ belong to $S$ and one has
\[
 q(\alpha)+q(\beta)
 \leq q(\alpha-e_i+e_j)+q(\beta+e_i-e_j).
\]
\end{definition}

On an $M$-convex support, it suffices to check this condition for pairs with $\|\alpha-\beta\|_1=\sum_i|\alpha_i-\beta_i|=4$ \cite[Theorem~6.4]{Murota}. For such a pair, we set $\gamma_i=\min\{\alpha_i,\beta_i\}$ for each $i$. Then $|\gamma|=d-2$, so $\alpha-\gamma$ and $\beta-\gamma$ belong to $S_\gamma$. Thus it is enough to check the exchange condition on each $S_\gamma$, after subtracting $\gamma$ from the indices.

Exponentiating an $M$-concave function gives the normalized coefficients of a Lorentzian polynomial \cite[Theorem~3.14 and Corollary~3.16]{BH}.
\begin{proposition}
\label{prop:mconcave-families}
Suppose that $S\subseteq\Delta_n^d$ is nonempty and $M$-convex, with $d\geq1$. If $q$ is an $M$-concave function on $S$, then for every $t\geq0$ the polynomial
\begin{equation}\label{eq:mconcave-degeneration}
 f_t(\mathbf x)=\sum_{\alpha\in S}
                 \frac{e^{tq(\alpha)}}{\alpha!}\,\mathbf x^\alpha
\end{equation}
belongs to $\cL(S)$.
\end{proposition}

\begin{proof}
For $t\geq0$, the logarithms of the normalized coefficients of $f_t$ are the values of the $M$-concave function $tq$. The assertion follows from \cite[Corollary~3.16]{BH}.
\end{proof}

\subsection{Bounded ratios and quadratic derivatives}

We adopt the notion of a bounded ratio from \cite[Definition~2.3]{CGLWW}.
\begin{definition}
\label{definition:bounded_ratios}
Suppose that $S\subseteq\Delta_n^d$ is nonempty and $M$-convex, with $d\geq1$. For $f\in\cL(S)$ and $w=(w_\alpha)_{\alpha\in S}\in\R^S$, we set
\[
 R_w(f)=\prod_{\alpha\in S}c_\alpha^{w_\alpha}.
\]
We call the vector $w$ a \emph{bounded ratio} on $\cL(S)$ if there is a constant $C>0$ such that $R_w(f)\leq C$ for every $f\in\cL(S)$. The set of all such vectors is denoted by $\BR(S)$. For full support, we write $\BR(n,d)=\BR(\Delta_n^d)$.
\end{definition}

The set $\BR(S)$ is a convex cone, since products and nonnegative powers of bounded ratios are again bounded ratios.

As in Part~I \cite[Section~3.2]{CGLWW}, we next construct bounded ratios on $\cL(S)$ from quadratic derivatives. Fix $d\geq2$ and $\beta\in\Delta_n^{d-2}$ such that the quadratic derivative $\partial^\beta f$ has nonempty support $S_\beta$. By \eqref{eq:quadratic-derivative-expansion}, the normalized coefficient indexed by $\eta$ in $\partial^\beta f$ is $c_{\beta+\eta}$. Accordingly, for $v\in\R^{S_\beta}$, define $\iota_\beta(v)\in\R^S$ by placing $v_\eta$ at the index $\beta+\eta$ for each $\eta\in S_\beta$ and zero elsewhere. Then, by Definition~\ref{definition:bounded_ratios}, one has
\begin{equation}
    \label{eq:shifted_coefficients}
    R_{\iota_\beta(v)}(f)=R_v(\partial^\beta f).
\end{equation}

We now use the shifted coefficients to define the cone generated by bounded ratios of quadratic derivatives.
\begin{definition}
For a nonempty $M$-convex set $S\subseteq\Delta_n^d$ with $d\geq2$, we define $K(S)\subseteq\R^S$ to be the cone generated by the vectors $\iota_\beta(v)$, where $\beta\in\Delta_n^{d-2}$, $S_\beta\ne\varnothing$, and $v\in\BR(S_\beta)$. For full support, we write $K(n,d)=K(\Delta_n^d)$.
\end{definition}

Since each nonzero quadratic derivative of a Lorentzian polynomial is Lorentzian \cite[Corollary~2.11]{BH}, the preceding identity \eqref{eq:shifted_coefficients} implies
\begin{equation}\label{eq:local-contained}
 K(S)\subseteq\BR(S).
\end{equation}

In degree two, $K(S)=\BR(S)$ follows directly from the definition. For $n=3$, the equality follows from Part~I \cite[Theorem~3.6]{CGLWW}, and the cases $n=1,2$ follow by viewing $f$ as a polynomial in three variables. Thus, we have the following proposition.
\begin{proposition}
\label{prop:easy-positive}
For every nonempty $M$-convex $S\subseteq\Delta_n^d$ with $d\geq2$, one has $\BR(S)=K(S)$ if $d=2$ or $n\leq3$.
\end{proposition}

\subsection{Necessary conditions for bounded ratios}

To compare $K(S)$ and $\BR(S)$ in the remaining cases, we first record two convex-analytic observations that yield necessary conditions for bounded ratios.

For $w\in\R^S$ and a function $q:S\to\R$, we can identify $q$ as a vector in $\R^S$ and write
\[
 \langle w,q\rangle=\sum_{\alpha\in S}w_\alpha q(\alpha)
\]
for the standard Euclidean product in $\R^S$.

The first condition comes from the families associated with $M$-concave functions. If $q$ is $M$-concave, by Proposition~\ref{prop:mconcave-families}, the polynomials $f_t$ defined in \eqref{eq:mconcave-degeneration} are Lorentzian with support $S$. For $w\in\BR(S)$, Definition~\ref{definition:bounded_ratios} therefore gives
\[
R_w(f_t)=e^{t\langle w,q\rangle}\leq C_w,
\]
for some constant $C_w>0$. Letting $t\to+\infty$, we obtain $\langle w,q\rangle\leq0$.

We derive the second condition using the \textit{scaling action} described in \cite[Section~1]{HHSW}. For fixed $f\in\cL(S)$ and $i\in[n]$, by \cite[Theorem~2.10]{BH}, we obtain a family of Lorentzian polynomials with support $S$
\[
g_t(\mathbf x)=f(x_1,\ldots,e^t x_i,\ldots,x_n)
\] 
for any $t\in \R$. For $v\in\BR(S)$, by Definition~\ref{definition:bounded_ratios}, one has
\[
 R_v(g_t)
 =R_v(f)e^{t\sum_{\alpha\in S}v_\alpha\alpha_i}
 \leq C_v
\]
for some constant $C_v>0$. Since $R_v(f)>0$, letting $t\to\pm\infty$ forces $\sum_{\alpha\in S}v_\alpha\alpha_i=0$. As $i$ is arbitrary, we obtain \textit{balanced identity}
\begin{equation}\label{eq:coefficient-balance}
 \sum_{\alpha\in S}v_\alpha\alpha=0.
\end{equation}
It follows that $\sum_{\alpha\in S}v_\alpha=0$ since $|\alpha|=d$. Moreover, for any linear function $\ell$ on $\R^n$ such that $\ell(\alpha)=a\cdot\alpha=\sum_{i=1}^na_i\alpha_i$, with $a\in\R^n$, \textit{balanced identity} gives $\langle v,\ell\rangle=0$.

We now explain how these two conditions can be used to prove inclusions between cones. For a convex cone $C\subseteq\R^S$, its \emph{polar cone} is defined by
\[
 C^\circ=\{q:S\to\R:\langle u,q\rangle\leq0
                    \text{ for every }u\in C\}
\]
in \cite[Section~14]{Rockafellar}. The first condition therefore says that every $M$-concave function belongs to $\BR(S)^\circ$.

Suppose that $C\subseteq\R^S$ is a finitely generated convex cone. Then $C$ is closed by \cite[Theorem~19.1]{Rockafellar}. For each $\omega\notin C$, the strong separation theorem \cite[Theorem~11.4]{Rockafellar} gives a function $q\in C^\circ$ such that $\langle \omega,q\rangle>0$. Thus, to prove $\BR(S)\subseteq C$, it suffices to show that $\langle \omega,q\rangle\le0$ for all $\omega\in\BR(S)$ and $q\in C^\circ$.

In particular, assume that $q$ can be written as a nonnegative linear combination of $M$-concave functions plus a linear function, that is, there exist $M$-concave functions $q_1,\ldots,q_m :S\to\R$, nonnegative scalars $\lambda_1,\ldots,\lambda_m$, and a vector $a\in\R^n$ such that
\[
q(\alpha)=\sum_{j=1}^m\lambda_j q_j(\alpha)+a\cdot\alpha.
\] 
Then, for every $\omega\in\BR(S)$, the two necessary conditions above give 
\[
\langle \omega,q\rangle
=\sum_{j=1}^m\lambda_j\langle \omega,q_j\rangle
+a\cdot\left(\sum_{\alpha\in S}\omega_\alpha\alpha\right)
\le0.
\]
Consequently, if every $q\in C^\circ$ can be written in this form, it is immediate that $\BR(S)\subseteq C$.

\subsection{Generators for quadratic ratios}
Let \(\{\varepsilon_\alpha\}_{\alpha\in S}\) be the standard basis of \(\mathbb R^S\). We adopt the notion of the \textit{triangular ratios} in \cite[Example~1.4]{HHSW} for the quadratic derivatives. For $\beta\in\Delta_n^{d-2}$ and pairwise distinct indices $i,j,k\in[n]$, we write
\begin{equation*}
    T_\beta(i,j,k)
   =\frac{c_{\beta+e_i+e_j}c_{\beta+2e_k}}
           {c_{\beta+e_i+e_k}c_{\beta+e_j+e_k}}, 
\end{equation*}
\begin{equation*}
     t_\beta(i,j,k)
   =\varepsilon_{\beta+e_i+e_j}+\varepsilon_{\beta+2e_k}
     -\varepsilon_{\beta+e_i+e_k}-\varepsilon_{\beta+e_j+e_k}.
\end{equation*}
We use these expressions only when all four displayed indices belong to $S$. Applying the inequality in \cite[Example~1.4]{HHSW} to the Hessian matrix $\mathcal{H}_{\beta} (f)$, we obtain $T_\beta(i,j,k)\leq2$. For simplicity, we write $T_i(j,k,\ell)= T_{e_i}(j,k,\ell)$ if $f$ is a cubic polynomial.

For $\beta\in\Delta_n^{d-2}$ and disjoint sets $I,J\subseteq[n]$ with $|I|=3$ and $|J|=2$, as in \cite[Theorem~A]{HHSW}, we define the exponent vector of a \textit{pentagonal ratio} by
\[
p_\beta(I,J)
=\sum_{\substack{i,j\in I\\i<j}}
  \varepsilon_{\beta+e_i+e_j}
+\sum_{\substack{s,t\in J\\s\le t}}
  \varepsilon_{\beta+e_s+e_t}
-\sum_{\substack{i\in I\\s\in J}}
  \varepsilon_{\beta+e_i+e_s}.
\] 
We use this expression only when all these indices belong to $S$.

Huang, Huh, Soskin, and Wang \cite[Theorems~A and~B]{HHSW} show that these vectors generate the cones of bounded ratios in four and five variables.
\begin{theorem}
\label{thm:quadratic-generators}
The cone $\BR(4,2)$ is generated by the vectors $t_0(i,j,k)$ with indices in $[4]$. The cone $\BR(5,2)$ is generated by the vectors $t_0(i,j,k)$ with indices in $[5]$ together with $p_0(I,J)$, where $[5]=I\sqcup J$, $|I|=3$, and $|J|=2$.
\end{theorem}

It follows that the vectors $t_\beta$ generate $K(4,d)$, while $t_\beta$ and $p_\beta$ generate $K(5,d)$.

\section{Equality for four-variable cubics}
\label{sec:quaternary-cubics}

For four-variable cubics with full support, the equality $\BR(4,3)=K(4,3)$ was established by Bathija, Rohatgi, and Soskin \cite[Theorem~8.3]{BRS}. In this section, we extend this equality to every nonempty $M$-convex support $S\subseteq\Delta_4^3$.
\begin{theorem}
\label{thm:quaternary-cubic-local-global}
For every nonempty $M$-convex $S\subseteq\Delta_4^3$, one has
\[
    \BR(S)=K(S).
\]
\end{theorem}

We have already established $K(S)\subseteq\BR(S)$ in \eqref{eq:local-contained}. We will prove the reverse inclusion using the strong separation theorem \cite[Theorem~11.4]{Rockafellar} together with the necessary conditions for bounded ratios in Section~\ref{sec:preliminaries}.

\subsection{Quadratic inequalities and cubic supports}

We first record three inequalities from \cite[Lemma~5.15]{BHKL2} for the entries of a quadratic Hessian.
\begin{lemma}
\label{lem:app-quadratic-inequalities}
Suppose that $P=(p_{ij})_{1\leq i,j\leq n}$ is symmetric and entrywise nonnegative, with at most one positive eigenvalue. Then the following hold.
\begin{enumerate}[label=\textit{(\roman*)}]
\item For pairwise distinct $i,j,k,\ell$, the numbers
\[
    \sqrt{p_{ij}p_{k\ell}},\qquad
    \sqrt{p_{ik}p_{j\ell}},\qquad
    \sqrt{p_{i\ell}p_{jk}}
\]
satisfy the triangle inequalities. If one is zero, the other two are equal.

\item For pairwise distinct $i,j,k$, one has
\[
    p_{ii}p_{jk}\leq 2p_{ij}p_{ik}.
\]

\item For $i\ne j$, one has
\[
    p_{ii}p_{jj}\leq p_{ij}^2.
\]
\end{enumerate}
\end{lemma}

We next use Lemma~\ref{lem:app-quadratic-inequalities} to study bounded ratios for quadratic Lorentzian polynomials with fixed support. For a nonempty $M$-convex set $T\subseteq\Delta_4^2$, we define $\mathcal E(T)\subseteq\R^T$ as follows. In each of the forms below, the variable indices that appear are pairwise distinct elements of $\{1,2,3,4\}$, and every basis vector $\varepsilon_\alpha$ that appears must have $\alpha\in T$.

We include all vectors of the forms
\begin{align}
 &\varepsilon_{2e_i}+\varepsilon_{2e_j}
 -2\varepsilon_{e_i+e_j},
 \label{eq:app-E-af}\\
 &\varepsilon_{2e_i}+\varepsilon_{e_j+e_k}
 -\varepsilon_{e_i+e_j}-\varepsilon_{e_i+e_k}.
 \label{eq:app-E-triangular}
\end{align}
For each choice of $i,j,k,\ell$ satisfying the conditions above, we also include both vectors
\begin{equation}\label{eq:app-E-fourpoint}
 \pm\bigl(
 \varepsilon_{e_i+e_j}+\varepsilon_{e_k+e_\ell}
 -\varepsilon_{e_i+e_k}-\varepsilon_{e_j+e_\ell}
 \bigr)
\end{equation}
whenever at least one of $e_i+e_\ell$ and $e_j+e_k$ is absent from $T$.

By Lemma~\ref{lem:app-quadratic-inequalities}, the ratios with exponent vectors \eqref{eq:app-E-af} and \eqref{eq:app-E-triangular} are bounded by $1$ and $2$, respectively. For \eqref{eq:app-E-fourpoint}, the support condition makes one of the three products vanish. The triangle inequalities then force the other two products to be equal, so both corresponding ratios are $1$. This proves $\mathcal E(T)\subseteq\BR(T)$.

We apply these bounded ratios to the quadratic derivatives of a cubic. The resulting exponent vectors generate the following cone.
\begin{definition}
\label{def:elementary-cone}
We write $S_i=S_{e_i}$ for $i\in[4]$ and define $K_0(S)$ to be the cone generated by $\iota_{e_i}(g)$ with $S_i\ne\varnothing$ and $g\in\mathcal E(S_i)$.
\end{definition}

The cone $K_0(S)$ is finitely generated, and every generator $u=\iota_{e_i}(g)$ satisfies $\sum_{\alpha\in S}u_\alpha\alpha=0$. Since $\mathcal E(S_i)\subseteq\BR(S_i)$ for every nonempty $S_i$, the definitions of $K_0(S)$ and $K(S)$ then give $K_0(S)\subseteq K(S)$. Together with \eqref{eq:local-contained}, this yields
\begin{equation}\label{eq:app-cone-chain}
    K_0(S)\subseteq K(S)\subseteq\BR(S).
\end{equation}

We next record a restriction on the support $S$ that follows from symmetric exchange. For $i\in[4]$, put $z_i=(1,1,1,1)-e_i$.
\begin{lemma}
\label{lem:app-support-reduction}
Suppose that $S_i$ contains all six points $e_j+e_k$ with $j<k$. Then either $S=e_i+T$ for a nonempty $M$-convex $T\subseteq\Delta_4^2$, or $z_1,z_2,z_3,z_4\in S$.
\end{lemma}

\begin{proof}
The hypothesis gives $z_j\in S$ for every $j\ne i$. If every $\alpha\in S$ satisfies $\alpha_i\geq1$, then $T=S-e_i$ is a nonempty $M$-convex subset of $\Delta_4^2$, so the first alternative holds. Otherwise, choose $\alpha\in S$ with $\alpha_i=0$. It remains to show that $z_i\in S$. We may assume that $\alpha\ne z_i$, so either $\alpha=2e_j+e_k$ with $i,j,k$ distinct, or $\alpha=3e_j$ with $j\ne i$.

In the first case, let $\ell$ be the remaining index. Symmetric exchange between $\alpha$ and $z_j=e_i+e_k+e_\ell$ at coordinate $j$ gives either $z_j-e_i+e_j=z_i\in S$ or $\alpha-e_j+e_\ell=z_i\in S$.

In the second case, symmetric exchange between $3e_j$ and $z_j$ at coordinate $j$ either gives $z_i\in S$ or produces $2e_j+e_k\in S$ with $k\notin\{i,j\}$, reducing to the first case.
\end{proof}

We now consider the case $S=e_i+T$ in Lemma~\ref{lem:app-support-reduction}. The following lemma relates the corresponding Lorentzian classes and their bounded ratios.
\begin{lemma}
\label{lem:app-translated-quadratic}
Suppose that $S=e_i+T$, where $T\subseteq\Delta_4^2$ is nonempty and $M$-convex. The operator $\partial_i$ gives a bijection from $\cL(S)$ to $\cL(T)$. For every $f\in\cL(S)$ and $\beta\in T$, the normalized coefficients satisfy $c_\beta(\partial_i f)=c_{\beta+e_i}(f)$. In particular, one has $\BR(S)=\iota_{e_i}(\BR(T))$.
\end{lemma}

\begin{proof}
The coefficient identity follows directly from differentiation. It shows that $\partial_i$ maps polynomials with support $S$ and positive coefficients bijectively onto those with support $T$ and positive coefficients.

For $j\ne i$, the polynomial $\partial_j f$ is either zero or a product of $x_i$ and a linear form with nonnegative coefficients. It is therefore Lorentzian by \cite[Corollary~2.32]{BH}. Since $S=e_i+T$ is $M$-convex, it follows from \cite[Definition~2.6 and Theorem~2.25]{BH} that $f\in\cL(S)$ if and only if $\partial_i f\in\cL(T)$. The assertion about bounded ratios follows from the coefficient identity above.
\end{proof}

\subsection{Decomposition into \texorpdfstring{$M$}{M}-concave functions}

Throughout the next three lemmas, we assume that $z_i\in S$ for every $i\in[4]$. For distinct $i,j\in[4]$, we write $a_{ij}=2e_i+e_j$ and $p_i=3e_i$. Together with the $z_i$, these are all the points of $\Delta_4^3$. For each $i\in[4]$, we put $N_i=\{j\in[4]:j\ne i,\ a_{ij}\in S\}$.

We first use symmetric exchange to determine the possible sizes of the sets $N_i$.
\begin{lemma}
\label{lem:app-support-graph}
For each $i\in[4]$, one has $|N_i|\in\{0,2,3\}$. If $p_i\in S$, then $|N_i|=3$.
\end{lemma}

\begin{proof}
Fix $i\in[4]$ and let $j,k,\ell$ be the remaining indices. If $a_{ij}\in S$, symmetric exchange with $z_j$ at coordinate $j$ gives $a_{ik}\in S$ or $a_{i\ell}\in S$. Thus $|N_i|\ne1$. Now suppose that $p_i\in S$. Fix $k\ne i$ and choose $j\notin\{i,k\}$. Apply symmetric exchange to $(z_j,p_i)$ at coordinate $k$. Since $i$ is the only coordinate at which $z_j$ is smaller than $p_i$, the exchange gives
\[
    p_i+e_k-e_i=a_{ik}\in S.
\]
Since $k\ne i$ was arbitrary, we have $|N_i|=3$.
\end{proof}

For a real-valued function $q$ on $S$ that vanishes at the $z_i$, we put $s_{ij}=-q(a_{ij})$ for $j\in N_i$ and $r_i=-q(p_i)$ whenever $p_i\in S$. The following lemma compares the linear inequalities obtained from $q\in K_0(S)^\circ$ with the exchange conditions for $M$-concavity.

\begin{lemma}
\label{lem:app-local-exchange}
Let $q$ be a real-valued function on $S$ satisfying $q(z_i)=0$ for every $i\in[4]$. If $q\in K_0(S)^\circ$, then for pairwise distinct $i,j,k$, one has
\begin{equation}\label{eq:app-polar-conditions}
 s_{ij}\geq0,\qquad
 s_{ij}\leq s_{ik}+s_{ji},\qquad
 s_{ij}+s_{ik}\leq r_i
\end{equation}
whenever all terms are defined. In this case, one also has
\begin{equation}\label{eq:app-D}
 s_{ij}=s_{ik}
\end{equation}
whenever $N_i=\{j,k\}$.

If $q$ satisfies \eqref{eq:app-polar-conditions} and \eqref{eq:app-D}, then $q$ is $M$-concave if and only if, for every $i$ with $|N_i|=3$, the minimum of the values $s_{ij}$ over $j\in N_i$ is attained at least twice.
\end{lemma}

\begin{proof}
Pairing $q$ with the generators arising from \eqref{eq:app-E-triangular} and using $q(z_i)=0$ gives \eqref{eq:app-polar-conditions}. If $N_i=\{j,k\}$, let $\ell$ be the remaining index. Since $e_i+e_\ell\notin S_i$, both signs of \eqref{eq:app-E-fourpoint} give $q(a_{ij})=q(a_{ik})$, which is \eqref{eq:app-D}.

Now suppose that $q$ satisfies \eqref{eq:app-polar-conditions} and \eqref{eq:app-D}. Since $S$ is $M$-convex, it suffices to check pairs at $\ell^1$-distance four \cite[Theorem~6.4]{Murota}. Subtracting $e_m$ from both points, for a suitable $m$, gives two points of $S_m$ with disjoint supports. We check the three cases in \cite[proof of Lemma~3.8]{BRS}.

For $2e_j,e_k+e_\ell$ with $j,k,\ell$ distinct, the exchange inequality is one of \eqref{eq:app-polar-conditions}. Both exchanged points belong to $S_m$ because $S$ is $M$-convex.

For $2e_j,2e_k$ with $j,k\ne m$, the exchange inequality follows from $s_{jm},s_{km}\geq0$. For $2e_m,2e_j$, one has $p_m\in S$, so $|N_m|=3$. For any $k\in N_m\setminus\{j\}$, one has
\[
 2s_{mj}\leq s_{mj}+s_{mk}+s_{jm}\leq r_m+s_{jm},
\]
as required.

In the remaining case, the original pairs are $a_{mj},z_j$ for $j\in N_m$, with $q(a_{mj})+q(z_j)=-s_{mj}$. Their exchanges are precisely the other pairs of this form with the same $m$. The exchange inequalities hold if and only if the minimum of the numbers $s_{mj}$ for $j\in N_m$ is attained at least twice. This is \eqref{eq:app-D} when $|N_m|=2$ and the stated condition when $|N_m|=3$. There is nothing to check when $N_m=\varnothing$.
\end{proof}

We now use the linear conditions above to express $q$ as a nonnegative linear combination of $M$-concave functions.
\begin{lemma}
\label{lem:app-star-decomposition}
Suppose that $q$ is a real-valued function on $S$ that vanishes at the $z_i$ and satisfies \eqref{eq:app-polar-conditions} and \eqref{eq:app-D}. Then $q$ is a finite nonnegative linear combination of $M$-concave functions on $S$.
\end{lemma}

\begin{proof}
For each nonempty $N_i$, let $m_i=\min_{j\in N_i}s_{ij}$ and put $u_{ij}=s_{ij}-m_i$ for $j\in N_i$. Set $m_i=0$ when $N_i=\varnothing$. These numbers are nonnegative, and $u_{ij}>0$ implies $|N_i|=3$ by Lemma~\ref{lem:app-support-graph} and \eqref{eq:app-D}. Whenever $a_{ij},a_{ji}\in S$, we have $u_{ij}\leq m_j+u_{ji}$. Indeed, if $u_{ij}>0$, choose $k\in N_i$ with $u_{ik}=0$ and apply $s_{ij}\leq s_{ik}+s_{ji}$.

We split the $m_i$ and $u_{ij}$ into summands, taking all unspecified entries in each summand to be zero. For each $i<j$ with $a_{ij},a_{ji}\in S$, take one summand contributing $\min\{u_{ij},u_{ji}\}$ to each of $u_{ij}$ and $u_{ji}$. Subtract this amount from both entries. Their difference is unchanged, and at least one is now zero. Thus the remaining $u_{ij}$ is at most $m_j$ whenever $a_{ji}\in S$. For each $j\in[4]$, put $m_j$ and all remaining $u_{ij}$ with $j\in N_i$ in one further summand.

For each summand, let $h$ vanish at every $z_i$ and set $h(a_{ij})=-m_i-u_{ij}$ for $j\in N_i$, using the values of $m_i$ and $u_{ij}$ in that summand. We specify $r_i=-h(p_i)$ below for $p_i\in S$. In each summand the bound $u_{ij}\leq m_j+u_{ji}$ holds whenever $a_{ij},a_{ji}\in S$, so $s_{ij}\leq s_{ik}+s_{ji}$ whenever required. For fixed $i$, at most one $u_{ij}$ is positive, and this can occur only when $|N_i|=3$. Thus \eqref{eq:app-D} holds, and when $|N_i|=3$ at least two of the $s_{ij}$ equal $m_i$.

For each $p_i\in S$, choose $k\in N_i$ for which the original $u_{ik}$ is zero. This entry is zero in every summand, so $s_{ik}$ is a minimum in each. Since $|N_i|=3$, we may take $r_i$ in each summand to be the sum of the other two entries. These values add up to the sum of the same two entries for $q$, which is at most $-q(p_i)$ by \eqref{eq:app-polar-conditions}. Increase $r_i$ in one summand by the difference. Each resulting function is $M$-concave by Lemma~\ref{lem:app-local-exchange}, and their sum equals $q$.
\end{proof}

By \eqref{eq:app-cone-chain}, it remains to prove the inclusion $\BR(S)\subseteq K(S)$.

\begin{proof}[Proof of Theorem~\ref{thm:quaternary-cubic-local-global}]
First, suppose that $S=e_i+T$ for some $i$ and $T$. Then $T=S_i$, so Lemma~\ref{lem:app-translated-quadratic} gives
\[
  \BR(S)=\iota_{e_i}\bigl(\BR(S_i)\bigr)\subseteq K(S).
\]
This proves the desired inclusion in this case. We may therefore assume that $S$ cannot be written as $e_i+T$ for any $i$ and $T$.

Suppose that $w\in\BR(S)\setminus K_0(S)$. Since $K_0(S)$ is finitely generated, separation gives $q\in K_0(S)^\circ$ with $\langle w,q\rangle>0$ \cite[Sections~11 and~19]{Rockafellar}. The family \eqref{eq:mconcave-degeneration} shows that $w$ pairs nonpositively with every $M$-concave function on $S$. Hence $q$ fails the exchange condition for some pair in $S$ at $\ell^1$-distance four \cite[Theorem~6.4]{Murota}.

Subtracting the coordinatewise minimum $e_i$ of this pair gives two points of $S_i$ with disjoint supports. They admit an exchange in $S_i$ since $S$ is $M$-convex. If one point were $2e_j$, the corresponding generator from \eqref{eq:app-E-af} or \eqref{eq:app-E-triangular} would give the exchange inequality. Thus the pair and its exchange give two of the three pairings of four distinct indices. If either point of the third pairing were missing from $S_i$, the two signs in \eqref{eq:app-E-fourpoint} would give equality for this exchange. Hence $S_i$ contains all six points $e_j+e_k$ with $j<k$.

Lemma~\ref{lem:app-support-reduction} now gives $z_j\in S$ for every $j$. Since the four $z_j$ form a basis of $\R^4$, choose a linear function $\ell$ with $\ell(z_j)=q(z_j)$ for every $j$. By \eqref{eq:coefficient-balance} for $w$ and the generators of $K_0(S)$, we have $q-\ell\in K_0(S)^\circ$ and $\langle w,q-\ell\rangle=\langle w,q\rangle>0$. Since $q-\ell$ vanishes at the $z_j$, Lemmas~\ref{lem:app-local-exchange} and \ref{lem:app-star-decomposition} express it as a nonnegative linear combination of $M$-concave functions. Its pairing with $w$ is therefore nonpositive, a contradiction. Thus $\BR(S)\subseteq K_0(S)\subseteq K(S)$ in this case.
\end{proof}

\section{Counterexamples in four and five variables}
\label{sec:minimal-obstructions}

We show that $\BR(\Delta_n^d) \nsubseteq K(n,d) $ for both $(n,d)=(4,4)$ and $(5,3)$ by constructing concrete examples. For each bound, we obtain two factorizations of the ratio and compare them using the inequality for four distinct indices in Lemma~\ref{lem:app-quadratic-inequalities}, as in \cite[proof of Proposition~2.4]{HHSW}.

For a nonempty proper subset $A\subsetneq[n]$, let $A^c=[n]\setminus A$. For $\alpha\in\N^n$, write
\[
\alpha(A):=\sum_{i\in A}\alpha_i.
\]
For $d\geq1$, we define a \textit{minimal cut function} $q_A$ on $\Delta_n^d \subset \N^n$ by
\[
 q_A(\alpha)=\min\{\alpha(A),\alpha(A^c)\}.
\]
The identity $|\alpha|=d$ gives $q_A=q_{A^c}$. For $d\geq2$, the function $q_A$ differs by a linear function from the negative of a sum of the functions in \cite[Definition~3.12 and Lemma~3.14]{BRS}. We prove it satisfies the symmetric exchange property, hence $M$-concave.

\begin{lemma}
\label{lem:cut-directions}
For every nonempty proper subset $A\subset[n]$ and every $d\geq1$, the function $q_A$ is $M$-concave. In particular, one has $q_A\in\BR(n,d)^\circ$.
\end{lemma}

\begin{proof}
Fix $\alpha,\beta\in\Delta_n^d$ and an index $i$ such that $\alpha_i>\beta_i$. Since $q_A=q_{A^c}$, we may assume that $i\in A$.

If there exists $j\in A$ with $\alpha_j<\beta_j$, then the exchange preserves the sum of the coordinates in $A$ for each vector, so
\[
    q_A(\alpha-e_i+e_j)+q_A(\beta+e_i-e_j)
    =q_A(\alpha)+q_A(\beta).
\]

Otherwise, $\alpha_k\geq\beta_k$ for every $k\in A$. Since $\alpha_i>\beta_i$, we have $\alpha(A)-1\geq\beta(A)$. As $|\alpha|=|\beta|=d$, there exists $j\in A^c$ with $\alpha_j<\beta_j$.

Let $h(t)=\min\{t,d-t\}$. By concavity, the increments $h(s+1)-h(s)$ are nonincreasing in $s$. Since $\alpha(A)-1\geq\beta(A)$, we obtain
\[
    h(\alpha(A))-h(\alpha(A)-1)
    \leq h(\beta(A)+1)-h(\beta(A)).
\]
Therefore,
\[
\begin{aligned}
    q_A(\alpha-e_i+e_j)+q_A(\beta+e_i-e_j)
    &=h(\alpha(A)-1)+h(\beta(A)+1)\\
    &\geq h(\alpha(A))+h(\beta(A))\\
    &=q_A(\alpha)+q_A(\beta).
\end{aligned}
\]
Thus $q_A$ is $M$-concave, and \eqref{eq:mconcave-degeneration} gives $q_A\in\BR(n,d)^\circ$.
\end{proof}

\subsection{Four-variable quartics}
\label{sec:four-quartic}

We omit commas and parentheses in multi-indices. For example, we write $c_{1030}$ for $c_{(1,0,3,0)}$. We consider the ratio
\begin{equation}
\label{eq:quartic-base-ratio}
 R(f)=\frac{c_{1030}c_{3010}c_{0202}}
                 {c_{1012}c_{1210}c_{2020}},
\end{equation}
whose exponent vector we denote by $w$. We first prove that $R$ is bounded and that $w$ lies outside $K(4,4)$. In Section~\ref{sec:propagation}, we apply the same ratio to suitable derivatives to obtain counterexamples in every degree $d\geq4$.
\begin{theorem}
\label{thm:four-variable-quartic}
For every $f\in\cL(4,4)$, one has $R(f)\leq64$. The exponent vector $w$ satisfies
\[
 w\in\BR(4,4)\setminus K(4,4).
\]
\end{theorem}

\begin{proof}
For the upper bound, we write
\[
 A=c_{1201}c_{0112},\qquad
 B=c_{1111}c_{0202},\qquad
 C=c_{1102}c_{0211}.
\]
The inequality for four distinct indices in Lemma~\ref{lem:app-quadratic-inequalities}, applied to the Hessian of $\partial_2\partial_4f$, gives $\sqrt B\leq\sqrt A+\sqrt C$. Consequently, $B\leq4\max\{A,C\}$, which is equivalent to $\min\{B/A,B/C\}\leq4$.

\Needspace{6\baselineskip}
The ratio $R$ has the two factorizations
\begin{align*}
 R(f)
 &=T_{1010}(2,4,3)\,T_{1010}(4,3,1)\,
   T_{1100}(2,4,3)\,T_{0011}(2,4,1)\,\frac BA,\\
 R(f)
 &=T_{1010}(2,4,3)\,T_{1010}(2,3,1)\,
   T_{0110}(2,4,1)\,T_{1001}(2,4,3)\,\frac BC.
\end{align*}
Each triangular ratio is bounded above by two, so both factorizations imply
\[
 R(f)\leq2^4\min\{B/A,B/C\}\leq64.
\]

To prove that $w\notin K(4,4)$, we set $q_0=q_{\{2,3\}}+q_{\{3,4\}}$ on $\Delta_4^4$ and define $q$ to agree with $q_0$ except that $q(2020)=q_0(2020)-2$.

The values of $q_0$ at the numerator indices $1030,3010,0202$ are $2,2,4$, and its values at the denominator indices $2020,1012,1210$ are $4,2,2$. It follows that $\langle w,q_0\rangle=0$ and $\langle w,q\rangle=2$.

Lemma~\ref{lem:cut-directions} and \eqref{eq:local-contained} give $q_0\in K(4,4)^\circ$. Only generators with a negative entry at $2020$ can have a larger pairing with $q$ than with $q_0$. Theorem~\ref{thm:quadratic-generators} reduces the verification to triangular generators in $\BR(4,2)$. A negative entry at $\alpha \in S$ must correspond to an off-diagonal entry in $\mathcal{H}_\beta(f)$ for some $\beta+e_i+e_j=\alpha$. The only off-diagonal Hessian entry with coefficient index $2020$ is the $(1,3)$-entry of $\mathcal{H}_{1010}(f)$. The functions $q_0$ and $q$ are invariant under interchanging coordinates $1$ and $3$, or coordinates $2$ and $4$, so it is enough to check $t_{1010}(1,2,3)$. Its pairings with $q_0$ and $q$ are $-2$ and $0$, respectively. Thus $q\in K(4,4)^\circ$, and $\langle w,q\rangle=2$ gives $w\notin K(4,4)$.
\end{proof}

\subsection{Five-variable cubics}
\label{sec:five-cubic}

We consider the ratio
\[
 U(f)=\frac{c_{30000}c_{21000}c_{10101}c_{00210}c_{00021}}
               {c_{20100}c_{20010}c_{20001}c_{01110}c_{00111}},
\]
whose exponent vector we denote by $v$.
\Needspace{8\baselineskip}
\begin{theorem}
\label{thm:five-variable-cubic}
For every $f\in\cL(5,3)$, one has $U(f)\leq32$. The exponent vector $v$ satisfies
\[
 v\in\BR(5,3)\setminus K(5,3).
\]
\end{theorem}

\begin{proof}
For the upper bound, we write
\[
 A=c_{21000}c_{10110},\qquad
 B=c_{20100}c_{11010},\qquad
 C=c_{20010}c_{11100}.
\]
Applying Lemma~\ref{lem:app-quadratic-inequalities} to the principal submatrix of $\mathcal{H}_{e_1}(f)$ indexed by $\{1,2,3,4\}$ gives $\sqrt A\leq\sqrt B+\sqrt C$. Hence one has $\min\{A/B,A/C\}\leq4$.

\Needspace{6\baselineskip}
The ratio $U$ has the two factorizations
\begin{align*}
 U(f)
 &=T_1(4,5,1)\,T_4(1,2,3)\,T_5(1,3,4)\,\frac AB,\\
 U(f)
 &=T_1(3,5,1)\,T_3(1,2,4)\,T_4(4,5,3)\,\frac AC.
\end{align*}
The bound of two for each triangular ratio now gives
\[
 U(f)\leq2^3\min\{A/B,A/C\}\leq32.
\]

We set $r_0=q_{\{2,3\}}+q_{\{2,4\}}+q_{\{2,5\}}$ on $\Delta_5^3$ and define $r$ to agree with $r_0$ except that $r(00111)=r_0(00111)-1$. At the numerator indices in $U$, the values of $r_0$ are $0,3,2,2,2$. At the denominator indices, they are $1,1,1,3,3$, in the displayed order. Both sums are nine, so $\langle v,r_0\rangle=0$ and $\langle v,r\rangle=1$.

As in the quartic case, one has $r_0\in K(5,3)^\circ$. To show that $r\in K(5,3)^\circ$, it suffices to check generators with a negative entry at $00111$, since every other pairing remains unchanged or decreases when $r_0$ is replaced by $r$. By Theorem~\ref{thm:quadratic-generators}, the generators are triangular or pentagonal. The coefficient $c_{00111}$ occurs only in $\mathcal{H}_3f$, $\mathcal{H}_4f$ and $\mathcal{H}_5f$. The function $r$ is invariant under permutations of coordinates $3,4,5$, so it is enough to consider $\mathcal{H}_3f$. In its Hessian, $c_{00111}$ is the $(4,5)$-entry.

Up to interchanging $4$ and $5$, the triangular generators with a negative entry there are $t_{e_3}(j,4,5)$ for $j=1,2,3$. For a pentagonal generator $p_{e_3}(I,J)$ to have a negative entry at this position, the indices $4$ and $5$ must lie in different parts $I,J$. After interchanging $4$ and $5$ if necessary, we therefore have $J=\{4,j\}$ for some $j\in\{1,2,3\}$. The pairings with $r_0$ and $r$ are as follows.
\begin{center}
\begin{tabular}{@{}lrr@{}}
\toprule
$g$&$\langle g,r_0\rangle$&$\langle g,r\rangle$\\
\midrule
$t_{e_3}(j,4,5)$, $j=1,2,3$&$-1$&$0$\\
$p_{e_3}(\{2,3,5\},\{1,4\})$&$-2$&$-1$\\
$p_{e_3}(\{1,3,5\},\{2,4\})$&$-3$&$-2$\\
$p_{e_3}(\{1,2,5\},\{3,4\})$&$-1$&$0$\\
\bottomrule
\end{tabular}
\end{center}
Every pairing is nonpositive, so $r\in K(5,3)^\circ$. The equality $\langle v,r\rangle=1$ then gives $v\notin K(5,3)$.
\end{proof}

\section{Extension to higher degrees and more variables}
\label{sec:propagation}
For Lorentzian polynomials $f\in \cL(4,d),\ d>4$, we reduce to the quartic case by taking a derivative of order $d-4$. Under factorial normalization, this shifts the coefficient indices, as in \cite[Proposition~9.5]{BRS}, where such shifts are used to compare Lorentzian polynomials with products of linear forms. For $d\geq4$, we set
\[
 a=\left\lfloor\frac{d-4}{2}\right\rfloor,\qquad
 b=\left\lceil\frac{d-4}{2}\right\rceil,\qquad
 \gamma=(a,0,b,0).
\]
Since $|\gamma|=d-4$, the polynomial $\partial^\gamma f$ is a quartic for every $f\in\cL(4,d)$. We apply the ratio $R$ from \eqref{eq:quartic-base-ratio} to this derivative and write
\begin{equation*}
 R_d(f)=R(\partial^\gamma f).
\end{equation*}
We denote its exponent vector in the normalized coefficients of $f$ by $w_d$.

To extend the separating function as well, we set
\[
 \alpha=\gamma+(2,0,2,0).
\]
Let $\delta=2$ when $d$ is even and $\delta=1$ when $d$ is odd. We also set $A=\{2,3\}$ and $B=\{3,4\}$. Using the functions defined before Lemma~\ref{lem:cut-directions}, we define $q_d$ on $\Delta_4^d$ to agree with $q_A+q_B$ except at $\alpha$, where its value is
\begin{equation*}
 q_d(\alpha)=q_A(\alpha)+q_B(\alpha)-\delta.
\end{equation*}
The bound for $R$ applies to $R_d$ because differentiation preserves Lorentzianity. We show that $q_d$ has nonpositive pairing with $K(4,d)$ and positive pairing with $w_d$ in every degree $d\geq4$.

\begin{theorem}
\label{thm:degree-obstruction-four}
For every $d\geq4$ and $f\in\cL(4,d)$, one has $R_d(f)\leq64$. The exponent vector $w_d$ and the function $q_d$ satisfy
\[
 w_d\in\BR(4,d)\setminus K(4,d),\qquad
 q_d\in K(4,d)^\circ,\qquad
 \langle w_d,q_d\rangle=\delta>0.
\]
\end{theorem}

\begin{proof}
The derivative $\partial^\gamma f$ is Lorentzian \cite[Corollary~2.11]{BH} and has full support. Theorem~\ref{thm:four-variable-quartic} therefore gives $R_d(f)\leq64$, so $w_d\in\BR(4,d)$.

Lemma~\ref{lem:cut-directions} and \eqref{eq:local-contained} give $q_A+q_B\in K(4,d)^\circ$. Decreasing its value at $\alpha$ can increase the pairing only with a generator having a negative entry there. Theorem~\ref{thm:quadratic-generators} reduces the verification to triangular generators, whose negative entries occur at off-diagonal Hessian positions. Since only coordinates $1$ and $3$ of $\alpha$ are nonzero, this position must be $(1,3)$, and the derivative index is
\[
 \beta=\alpha-e_1-e_3=(a+1,0,b+1,0).
\]
Since $t_\beta(i,j,k)=t_\beta(j,i,k)$, it suffices, for each $r\in\{2,4\}$, to check the two generators
\[
 t_\beta(1,r,3)\quad\text{and}\quad t_\beta(r,3,1).
\]
Both have entry $-1$ at $\alpha$.

For $t_\beta(1,2,3)$ and $t_\beta(4,3,1)$, the pairing with $q_A$ is zero. For the other two generators, the pairing with $q_B$ is zero. Since $\beta(A)=\beta(B)=b+1$, each remaining nonzero pairing is the sum of the values of $\min\{s,d-s\}$ at $b+1$ and $b+3$ minus twice its value at $b+2$. At $b+1,b+2,b+3$, this function takes the values $k-1,k,k-1$ if $d=2k$ and $k,k,k-1$ if $d=2k+1$. Thus every pairing under consideration is $-\delta$ before the modification at $\alpha$ and zero afterwards. All other pairings remain nonpositive. This proves $q_d\in K(4,d)^\circ$.

For each of $A$ and $B$, the coordinate sums at the numerator indices of $R_d$ are $b+1,b+2,b+3$, in some order. The same three values occur at the denominator indices, so $\langle w_d,q_A+q_B\rangle=0$. The index $\alpha$ occurs once in the denominator. Consequently, we have $\langle w_d,q_d\rangle=\delta>0$, which gives $w_d\notin K(4,d)$.
\end{proof}

\Needspace{8\baselineskip}
It remains to extend the counterexamples to more variables. For the induction step, we place the bounded ratio on the coordinate face where the added variables vanish. To extend the separating function, we sum the added coordinates with one of the original coordinates. A related use of coordinate faces appears in \cite[Lemma~9.2]{BRS}.

\begin{lemma}
\label{lem:variable-aggregation}
Suppose that $m,d\geq2$ and that the vectors $w\in\BR(m,d)$ and $u\in K(m,d)^\circ$ satisfy $\langle w,u\rangle>0$. For every $n\geq m$, there exist $W\in\BR(n,d)$ and $q\in K(n,d)^\circ$ such that $\langle W,q\rangle>0$. In particular, $W$ does not belong to $K(n,d)$.
\end{lemma}

\begin{proof}
We identify $\Delta_m^d$ with the face of $\Delta_n^d$ on which the last $n-m$ coordinates vanish and extend $w$ by zero to obtain $W$. For every $f\in\cL(n,d)$, the substitution $x_{m+1}=\cdots=x_n=0$ gives a polynomial in $\cL(m,d)$ with the same coefficients on this face \cite[Theorem~2.10]{BH}. The bound for $R_w$ therefore gives $W\in\BR(n,d)$.

For $\alpha\in\N^n$, we define
\[
 \pi(\alpha)=
 (\alpha_1,\ldots,\alpha_{m-1},\alpha_m+\cdots+\alpha_n) \in \N^m
\]
and set $q(\alpha)=u(\pi(\alpha))$ for $\alpha\in\Delta_n^d$.

To prove that $q\in K(n,d)^\circ$, we take $v\in\BR(n,2)$. For each $\xi\in\Delta_m^2$, we define $V_\xi$ to be the sum of $v_\eta$ over all $\eta\in\Delta_n^2$ with $\pi(\eta)=\xi$. For every $h\in\cL(m,2)$, the polynomial
\[
 \widetilde h(\mathbf x)
 =h(x_1,\ldots,x_{m-1},x_m+\cdots+x_n)
\]
is Lorentzian and has full support \cite[Theorem~2.10]{BH}. The multinomial formula and the factorial normalization show that its normalized coefficient at $\eta$ equals that of $h$ at $\pi(\eta)$. Hence $R_v(\widetilde h)=R_V(h)$, and boundedness of $R_v$ implies $V\in\BR(m,2)$.

For $\beta\in\Delta_n^{d-2}$, we have $\pi(\beta+\eta)=\pi(\beta)+\pi(\eta)$. Thus shifting $v$ by $\beta$ and then adding the entries with the same image under $\pi$ gives $V$ shifted by $\pi(\beta)$. Since $\iota_{\pi(\beta)}(V)$ belongs to $K(m,d)$ by hypothesis, its pairing with $u$ equals that of $\iota_\beta(v)$ with $q$ and is therefore nonpositive. Since every element of $K(n,d)$ is a finite sum of such vectors, we conclude that $q\in K(n,d)^\circ$. On this coordinate face, $\pi$ recovers the first $m$ coordinates, so one has
\[
 \langle W,q\rangle=\langle w,u\rangle>0.
\]
\end{proof}

We now prove Theorem \ref{thm:main-classification} with preceding lemmas.
\begin{proof}[Proof of Theorem~\ref{thm:main-classification}]
Proposition~\ref{prop:easy-positive} gives $\BR(S)=K(S)$ when $d=2$ or $n\leq3$, and Theorem~\ref{thm:quaternary-cubic-local-global} gives the case $(n,d)=(4,3)$.

For every remaining pair $(n,d)$, either $d=3$ and $n\geq5$, or $d\geq4$ and $n\geq4$. In the first case, we apply Lemma~\ref{lem:variable-aggregation} to the exponent vector and separating function constructed in the proof of Theorem~\ref{thm:five-variable-cubic}. In the second case, we apply the same lemma to $w_d$ and $q_d$ from Theorem~\ref{thm:degree-obstruction-four}. We obtain in both cases a bounded ratio whose exponent vector lies outside $K(n,d)$. Together with \eqref{eq:local-contained}, this gives $K(n,d)\subsetneq\BR(n,d)$.
\end{proof}

\vspace{0.5\baselineskip}

\noindent\textbf{Declaration of AI use.}
The authors used ChatGPT to help find the counterexamples in this paper and prove Theorem \ref{thm:quaternary-cubic-local-global}. They verified the counterexamples and take responsibility for the content of the manuscript.

\end{document}